\documentclass[11pt]{article}
\usepackage{amsmath}
\usepackage{amssymb}
\usepackage{amscd}
\usepackage{amsmath}
\usepackage[all]{xy}
\usepackage{setspace}
\usepackage{graphics,epsfig}
\usepackage{graphicx}
\usepackage{float}
\usepackage{stackrel}
\usepackage{tikz-cd}
\usepackage{tikz}
\usetikzlibrary{positioning}
\usepackage{amsthm}
\DeclareMathAlphabet{\eusm}{OT1}{eusm}{m}{n}
\newtheorem{thm}{Theorem}[section]

\newtheorem{cor}[thm]{Corollary}
\newtheorem{exam}[thm]{Example}

\newtheorem{lem}[thm]{Lemma}

\def\vsp{\vspace{1ex}}

\begin{document}
\begin{center}
{\rm {\LARGE A bijection between the indecomposable summands of two multiplicity free tilting modules}}
\end{center}
\begin{center}
Dedicated to the memory of Andrzej Skowroński.
\end{center}
\vsp

\begin{center}
{\rm {\Large Gabriella D$^{'}$Este$^a$ and H. Melis Tekin Akcin$^b$
}}\\
${}^a$ Department of Mathematics, University of Milano,\\
Milano, Italy\\
e-mail: gabriella.deste@unimi.it\\

${}^b$ Department of Mathematics, Hacettepe University,\\
06800 Beytepe, Ankara, Turkey\\
e-mail: hmtekin@hacettepe.edu.tr\\
\end{center}

\vsp
\begin{abstract}

We show that there is a reflection type bijection between the indecomposable summands of two multiplicity free tilting modules $X$ and $Y$. This bijection fixes the common indecomposable summands of $X$  and  $Y$ and sends indecomposable projective (resp., injective) summands of exactly one module to non-projective (resp., non-injective) summands of the other. Moreover, this bijection interchanges the two possible non-isomorphic complements of an almost complete tilting module.\\



\end{abstract}

\section{Introduction}
The word "reflection" appears in the title of the important paper \cite{BB} by Brenner and Butler which is the origin of the tilting theory. In
this paper we investigate bijections (between the indecomposable summands of multiplicity free tilting modules of finite dimension) which
behave like "reflections". For instance, let $A$ be the $K$-algebra $\left( \begin{array}{cc} K & 0 \\ K & K \\ \end{array} \right)$ given by
the quiver
\begin{tikzcd}
\begin{array}{c}\bullet \\ 1 \end{array} \arrow[r, shift left] & \begin{array}{c} \bullet \\ 2 \end{array}
\end{tikzcd}. Then the regular module $X=\begin{array}{c}1 \\ 2 \end{array}\oplus 2$ and the minimal injective cogenerator $Y=\begin{array}{c}
1 \\ 2 \end{array}\oplus 1$ are multiplicity free tilting modules. Since the Auslander-Reiten quiver is of the form
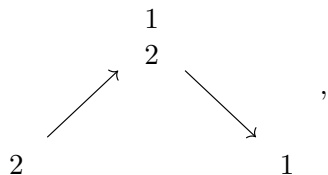
\begin{center}
$\begin{tikzcd}
                                                     & \begin{array}{c} 1\\2\\ \end{array} \arrow[rd] &
                                                     \\
\begin{array}{c} 2 \end{array} \arrow[ru] &                                                        & \begin{array}{c}1
\end{array}
\end{tikzcd}
$,
\end{center} we obtain $Y$ from $X$ by making a reflection along the vertical line passing through the vertex $\begin{array}{c} 1 \\ 2
\end{array}$. In this way, we exchange the end terms $2$ and $1$ of the Auslander-Reiten sequence
\begin{tikzcd}
0 \arrow[r] & 2 \arrow[r] & \begin{array}{c} 1\\2 \end{array} \arrow[r] & 1 \arrow[r] & 0
\end{tikzcd}. The observation of many examples suggests that this is not a particular case, as we will see in Section 2 (Theorem \ref{Theorem4}). The exact sequences in play are non-split exact sequences, but not necessarily Auslander-Reiten sequences (Example \ref{Ex1}(ii)). Moreover, our bijections
between indecomposables are not always unique (Example \ref{Ex1}(ii)). We will see (Examples \ref{Ex2} and \ref{Ex3}) other examples, where there is exactly one bijection satisfying our hypotheses and the two tilting modules are neither projective nor injective, as in the example described by the previous picture. 

Throughout this note, we assume that $K$ is a field, $A$ is a finite dimensional $K$-algebra and almost all the modules are finite dimensional left $A$-modules unless otherwise stated. For any module
$M$, we denote by $pdimM$ and $idimM$ the projective and the injective dimension of $M$, respectively while we denote by $addM$ the class of
all finite direct sums of direct summands of $M$. We denote by $Gen(M)$ (resp., $Cogen(M)$) the class of all modules generated (resp., cogenerated) by $M$. We will use the definition of tilting module given in \cite[page 167]{Ringel}. If $T$ is a finite
dimesional $A$-module, then we say that $T$ is a \emph{tilting module} if the following conditions hold:
\begin{enumerate}
  \item [($\alpha$)] $pdimT\leq 1$;
  \item [($\beta$)] Ext$_{A}^{1}(T,T)=0$;
  \item [($\gamma$)] The number of isomorphism classes of the indecomposable summands of $T$ is equal to the number of simple $A$-modules.
\end{enumerate}
Moreover, we say that $T$ is a \emph{partial tilting module} if $T$ satisfies conditions $(\alpha)$ and $(\beta)$. We know from \cite[Corollary 6.2]{HR3} that a partial tilting module is a tilting module if and only if the following condition holds:

\begin{enumerate}
\item [($\gamma'$)] There is a short exact sequence of the form $0\rightarrow A \rightarrow T' \rightarrow T''\rightarrow 0$, where $T'$ and $T''$ are in $addT$.
\end{enumerate}
Finally, we say that $T$ is a
\emph{cotilting module} if $T$ satisfies conditions $(\alpha^{*})$, $(\beta)$ and $(\gamma)$, where $(\alpha^{*})$ is the following condition:

\begin{enumerate}
\item [($\alpha^{*}$)]~idim$T\leq 1$.
\end{enumerate}
We know from Bongartz's lemma \cite{Bongartz} that any partial tilting module is a direct summand of a tilting module. Let $n>1$ be the number of simple modules, $T$ be a partial tilting module with exactly $n-1$ indecomposable direct summands (up to isomorphism) and $X$ be an indecomposable module which is not isomorphic to a direct summand of $T$. Then we say that $T$ is an {\em almost complete tilting module} and $X$ is called {\em a complement of $T$} if $T\oplus X$ is a tilting module. We know from \cite[Corollary 1.2]{Happel} that an almost complete tilting module has either one or two indecomposable complements. Moreover, \cite[Theorem 1.1]{Happel} if $X$ and $Y$ are two non-isomorphic complements of $T$ such that $Ext^{1}(Y,X)\neq 0$, then there is an exact sequence $0\rightarrow X\rightarrow E \rightarrow Y\rightarrow 0$ with $E\in addT.$
A "more regular" \cite[Abstract]{Buan} and "more symmetric" \cite[page 209]{T} behaviour appears 
in the cluster category of a finite dimensional hereditary algebra \cite[Theorem 5.1]{Buan}. Indeed, in this new and larger category, given an object which is the analogue of an almost complete tilting module, "there are always two ways of completing such an object" \cite[Theorem 2.4]{T}. The main result of this paper (Theorem  \ref{Theorem4}) describes a similar situation. In fact, a module which is not fixed under our bijections is associated to its image by a mutation similar to the mutation between two non-isomorphic complements of an almost complete tilting module.

In a forthcoming paper, we will investigate whether the indecomposable summands of two "non-classical" tilting modules, that is tilting modules of projective dimension $>1$ \cite{HR3} or $\tau$-tilting modules \cite{Iyama} are involved in a bijection similar to that described in this paper.   

If $A$ is the algebra given by a quiver $Q$ with vertices $1,\ldots, n$, then the symbols $1, \ldots, n$ will denote also the simple modules
corresponding to the vertices $1,\ldots, n$. Finally, with the usual conventions (\cite{ARS} or \cite{S}) pictures of the form

$$
\begin{array}{rlc}
 1\\
 2\\
 3\\
 4\\
 1
\end{array}
\quad,
\begin{array}{rlc}
 ~~1~~\\
 2~3~4
\end{array}
\quad,
\begin{array}{rlc}
 1~2~3\\
 ~~4~~
\end{array}
\quad,
\begin{array}{rlc}
 ~1~\\
 2~3\\
 ~4~
\end{array}
$$
describe indecomposable modules. In section 2, we prove the main results of this paper. We collect all the examples in section 3.

\section{Results}

We begin with three lemmas.
\begin{lem}\label{Lemma1}\cite[page 167]{Ringel} 
Let $A$ be a finite dimensional $K$-algebra and let $T$ and $S$ be finite dimensional $A$-modules
such that $T$ is a tilting module and $T\oplus S$ is a partial tilting module. Then we have $S\in addT$.
\end{lem}

\begin{lem}\label{Lemma2}
Let $A$ be a finite dimensional $K$-algebra. Let $X=X_1\oplus X_2 \oplus \ldots \oplus X_{n}$ and 
$Y=Y_1\oplus Y_2 \oplus \ldots \oplus Y_{n}$ be finite dimensional multiplicity free tilting modules with $X_{i}$ and $Y_{i}$
indecomposable for all $i=1,\ldots, n$. Let $F(i)$ denote the following subset of $\{1,\ldots ,n\}$ for any $i$:
\begin{center}
$F(i)= \{j \mid \text{either } X_{i}\simeq Y_{j} \text{ or } Ext^1_{A}(X_{i},Y_{j})\oplus 
Ext^1_{A}(Y_{j},X_{i})\neq 0 \}$.
\end{center}
The following facts hold:
\begin{itemize}
\item[(a)] $F(i)\neq \emptyset$ for all $i=1,\ldots, n$.
\item[(b)]If $1\leq i_1< i_2 < \ldots <i_{m}\leq n$ and $m\geq 2$, then we have 
$\vert F(i_1)\cup \ldots \cup F(i_m)\vert\geq m$.
\end{itemize}
\end{lem}

\begin{proof}
Assume on the contrary that $F(i)=\emptyset$ for some $i\in \{1,\ldots, n\}$. Let $U=X_{i}\oplus Y= X_{i}\oplus Y_1 \oplus \ldots \oplus Y_{n}.$ Then we have $pdimU \leq 1$. Moreover, the assumption $F(i)=\emptyset$ implies that $X_{i}\not\simeq Y_{j}$ and $Ext^1_{A}(X_{i}, Y_{j})\oplus Ext^1_{A}(Y_{j},X_{i})=0$ for any $j$.
It follows that 
\begin{itemize}
\item[(1)] $U$ is a multiplicity free partial tilting module.  
\end{itemize}
Since $Y$ is a tilting module, Lemma \ref{Lemma1} implies that $X_{i}\in addY$ and so
\begin{itemize}
\item[(2)] $U$ is not multiplicity free. 
\end{itemize}
This contradiction implies that $(a)$ holds.\\ 
Without loss of generality, we may assume $i_{t}=t$ for any $t\leq m$. Suppose 
by contradiction that $\vert F(1)\cup \ldots \cup F(m)\vert< m$. To simplify the notation
we may assume that $F(1)\cup \ldots \cup F(m)$ is a subset of $\{1,\ldots,m-1\}$.
Let $W=X_1\oplus \ldots \oplus X_{m}\oplus Y_{m}\oplus \ldots \oplus Y_{n}$. Then we have 
$pdimW\leq 1$. Moreover, for any $i\leq m$ and $j\geq m$, we have $X_{i}\not\simeq Y_{j}$
and $Ext^1_{A}(X_{i}, Y_{j})\oplus Ext^1_{A}(Y_{j},X_{i})=0$. Since $X$ and $Y$ are multiplicity free 
tilting modules, it follows that 
\begin{itemize}
\item[(3)] $W$ is a multiplicity free partial tilting module. 
\end{itemize}
To find a contradiction, let $M=M_1\oplus \ldots \oplus M_{n+1}$ be 
a partial tilting module such that $M_1, \ldots, M_{n+1}$ are indecomposable. 
Let $L=M_1 \oplus \ldots \oplus M_{n}$. Then one of the following cases occurs:
\begin{itemize}
\item[(4)] $L$ is a tilting module.
\item[(5)] $L$ is not multiplicity free.
\end{itemize} 
If $(4)$ holds, then we deduce from Lemma \ref{Lemma1} that $M_{n+1}\in addL$. This remark 
and $(5)$ imply that in both cases $M$ is not multiplicity free. This implies that 
\begin{itemize}
\item[(6)] $W$ is not multiplicity free.
\end{itemize}
This contradiction implies that $(b)$ holds.

\end{proof}

\begin{lem}\label{Lemma3}
Let $m$, $n$ be natural numbers with $1\leq m <n$. Let $F(1),\ldots, F(n)$ be subsets of $\{1,\ldots, n\}$ such that $F(i)\neq \emptyset$ for all $i=1,\ldots, n$ and $\vert F(i_1)\cup \ldots \cup F(i_k)\vert\geq k$ for any $1\leq i_1< i_2 < \ldots <i_{k}\leq n$ and $k\leq n$. Then there is an injective map
\begin{center}
$h: \{1,\ldots,m+1\}\rightarrow \{1,\ldots,n\}$
\end{center}
such that $h(i)\in F(i)$ for $i \leq m+1$.
\end{lem}

\begin{proof}
We proceed by induction on $m$. Assume $m=1$. Since $\vert F(1)\cup F(2)\vert \geq 2$, there are two distinct elements $x$ and $y$ such that $x\in F(1)$ and $y\in F(2)$. Hence, the map such that $1\mapsto x, 2\mapsto y$ has the desired property.\\ 
Without loss of generality, assume that $i\in F(i)$ for any $i\leq m.$ If there exists some 
$x\in F(m+1)\setminus \{1, \ldots, m\}$, then the map 
$h: \{1,\ldots, m+1\} \rightarrow \{1,\ldots, n\}$ such that $m+1 \mapsto x$ and $i\mapsto i$ for 
$i\leq m$ is injective.\\ 
Assume now $F(m+1)\subseteq \{1, \ldots, m\}$. Without loss of generality, we may assume\\
$(1)$ $F(m+1)=\{1, \ldots, r\}$ for some $r\leq m$.\\
Suppose there exists $x\in F(i)\setminus \{1, \ldots, m\}$ for some $i\leq r$. In this case, $h$ may be the map such that\\
$(2)$ $m+1\mapsto i$, $i\mapsto x$ and $j\mapsto j$ for any $j\leq m$ and $j\neq i$.\\
Assume now that\\ 
$(3)$ $F(1)\cup \ldots \cup F(r)\subseteq \{1, \ldots, m\}$.\\
Then we deduce from $(1)$ and $(3)$ that $F(1)\cup\ldots\cup F(r)$ is a subset of $\{1, \ldots, m\}$ with at least $r+1$ elements. Hence, we have $m\geq r+1$.
Choose some 
\begin{center}
$l_1\in (F(1)\cup \ldots \cup F(r))\setminus \{1, \ldots, r\}$.
\end{center}
Assume that $l_1 \in F(i)$ for some $i\leq r$. If there exists some $x\in F(l_1)\setminus \{1,\ldots, m\}$, then there is an injective map $h: \{1,\ldots,m+1\}\rightarrow \{1,\ldots,n\}$ such that \\
$(4)$ $m+1\mapsto i$, $i\mapsto l_1$, $l_1\mapsto x$ and $j\mapsto j$ for any $j\leq m$ and $j\notin \{i,l_1\}$.\\
Assume now that\\ 
$(5)$ $F(l_1)\subseteq \{1, \ldots, m\}.$\\
Then we deduce from $(1)$, $(3)$ and $(5)$ that $F(1)\cup \ldots \cup F(r)\cup F(l_1)$ is a subset of $\{1, \ldots, m\}$ with at least $r+2$ elements. Hence, we have $m\geq r+2$. Fix some
 \begin{center}
$l_2\in (F(1)\cup \ldots \cup F(r)\cup F(l_1))\setminus \{1,\ldots, r,l_1\}$.
\end{center}  
Suppose first that there is some  $x\in F(l_2)\setminus \{1,\ldots, m\}$. 
If $l_2\in F(l_1)$, then $h$ may be the map such that\\
$(6)$ $m+1\mapsto i$, $i\mapsto l_1$, $l_1\mapsto l_2$, $l_2 \mapsto x$ and $j\mapsto j$ for any $j\leq m$ and $j\notin \{i,l_1, l_2\}$.\\ 
If $l_2\in F(s)$ for some $s\leq r$, then $h$ may be the map such that \\
$(7)$ $m+1\mapsto s$, $s\mapsto l_2$, $l_2\mapsto x$ and $j\mapsto j$
for any $j\leq m$ and $j\notin \{s,l_2\}$.\\
Suppose finally that \\
$(8)$ $F(l_2)\subseteq \{1,\ldots, m\}$.\\
Then we deduce from $(1)$, $(3)$, $(5)$ and $(8)$ that $F(1)\cup \ldots \cup F(r)\cup F(l_1)\cup F(l_2)$ is a subset of $\{1, \ldots, m\}$ with at least $r+3$ elements. Consequently, we have $m\geq r+3$.  
Fix some 
\begin{center}
$l_3\in (F(1)\cup \ldots \cup F(r)\cup F(l_1)\cup F(l_2))\setminus \{1,\ldots,r,l_1,l_2\}$.
\end{center} 
Suppose there exists some $x\in F(l_3)\setminus \{1,\ldots,m\}$. If $l_3\in F(s)$ 
for some $s\leq r$, then $h$ may be the map such that \\
$(9)$ $m+1\mapsto s$, $s\mapsto l_3$, $l_3\mapsto x$
and $j\mapsto j$ for any $j\leq m$ and $j\notin \{s,l_3\}$.\\ 
If $l_3\in F(l_1)$, then $h$ may be the map such that \\
$(10)$ $m+1\mapsto i$, $i\mapsto l_1$, $l_1\mapsto l_3$, $l_3\mapsto x$ and $j\mapsto j$ for any $j\leq m$ and $j\notin \{i,l_1,l_3\}$.\\
Suppose now $l_3\in F(l_2)$. If $l_2\in F(s)$ for some $s\leq r$, then $h$ may be the map such that \\
$(11)$ $m+1\mapsto s$, $s\mapsto l_2$, $l_2\mapsto l_3$, $l_3\mapsto x$ and $j\mapsto j$ for any $j\leq m$ and $j\notin \{s,l_2,l_3\}$.\\
If $l_2\in F(l_1)$, then $h$ may be the map such that \\
$(12)$ $m+1\mapsto i$, $i\mapsto l_1$, $l_1\mapsto l_2$, $l_2\mapsto l_3$ $l_3\mapsto x$ and $j\mapsto j$ for any $j\leq m$ and $j\notin \{i,l_1,l_2,l_3\}$.\\
Suppose next that\\
$(13)$ $F(l_3)\subseteq \{1,\ldots, m\}$.\\
Then we deduce from $(1)$, $(3)$, $(5)$, $(8)$ and $(13)$ that $F(1)\cup \ldots \cup F(r)\cup F(l_1)\cup F(l_2)\cup F(l_3)$ is a subset of $\{1, \ldots, m\}$
with at least $r+4$ elements. Hence, we have $m\geq r+4$.
Consequently, after finitely many steps, we can find elements $l_1,\ldots, l_t$ such that 
$l_1\in (F(1)\cup \ldots \cup F(r))\setminus \{1,\ldots,r\}$, $l_{j}\in (F(1)\cup \ldots \cup F(r))\cup F(l_1)\cup \ldots\cup F(l_{j-1})\setminus \{1,\ldots,r, l_1, \ldots, l_{j-1}\}$ for any $j=2\ldots, t-1$, $F(l_{t})\nsubseteq \{1, \ldots, m\}$ and $F(1)\cup \ldots \cup F(r))\cup F(l_1)\cup \ldots\cup F(l_{t-1})\subseteq \{1, \ldots, m\}$. These results imply that there is an injective map $h: \{1,\ldots,m+1\}\rightarrow \{1,\ldots,n\}$. The proof of the lemma is completed. 
\end{proof}
\begin{exam}
There exist subsets $F(1),\ldots, F(4)$ of $\{1,2,3,4\}$ and an injective map
\begin{center}
$g: \{1,2,3\}\rightarrow \{1,2,3,4\}$
\end{center}
with the following properties:
\begin{itemize}
\item[(i)] $F(1),\ldots, F(4)$ satisfy the hypotheses of Lemma \ref{Lemma2} and $g(i)\in F(i)$
for any $i=1,2,3$.
\item[(ii)] There is a unique permutation $s$ of $\{1,2,3,4\}$ such that $s(i)\in F(i)$ for any $i$
and we have $s(i)\neq g(i)$ for any $i=1,2,3.$
\end{itemize}
\end{exam}
\noindent
{\bf Construction:}
Let $F(1)=\{1,2\}$, $F(2)=\{2,3\}$, $F(3)=\{3,4\}$, $F(4)=\{1\}$.
Next, let $g$ be the map such that $g(i)=i$ for $i=1,2,3$. Then clearly $(i)$ holds. Moreover, $s$ is the map such that $4\mapsto 1, 1\mapsto 2, 2\mapsto 3, 3\mapsto 4$. Hence, also $(ii)$ holds. 
$\square$

\begin{thm}\label{Theorem4}
Let $A$ be a finite dimensional $K$-algebra and let $n$ be a natural number. Let $X=X_1\oplus \ldots \oplus X_n$ and $Y=Y_1\oplus \ldots \oplus Y_n$ be finite dimensional multiplicity free tilting modules with $X_{i}$ and $Y_{i}$ indecomposable for all $i=1,\ldots, n$. Then there exists a permutation $s$ of $\{1,\ldots, n\}$ such that for any $i=1,\ldots, n$ either $X_{i}\simeq Y_{s(i)}$ or there exists a non-split exact sequence of the form 
\begin{center}
$0\rightarrow U \rightarrow V \rightarrow W \rightarrow 0$
\end{center}
with $\{U,W\}=\{X_{i},Y_{s(i)}\}$.
\end{thm}

\begin{proof}
Let $F(1), \ldots F(n)$ be subsets of $\{1,\ldots,n\}$ defined in Lemma \ref{Lemma2} by the formula \\ 
$(1)\ F(i)= \{j \mid \text{either } X_{i}\simeq Y_{j} \text{ or } Ext^1_{A}(X_{i},Y_{j})\oplus 
Ext^1_{A}(Y_{j},X_{i})\neq 0 \}$.\\
Then Lemma \ref{Lemma2} guarantees that these subsets satisfy the properties of the subsets of Lemma \ref{Lemma3}. This observation and Lemma \ref{Lemma3} imply that \\
$(2)$ There is a permutation $s$ of $\{1,\ldots, n\}$ such that $s(i)\in F(i)$ for all $i=1,\ldots, n$.\\ Putting together $(1)$ and $(2)$, we conclude that either $X_{i}\simeq Y_{s(i)}$ or one of the vector spaces $Ext^1_{A}(X_{i}, Y_{s(i)})$ and  $Ext^1_{A}(Y_{s(i)},X_{i})$ is different from $0$. Therefore, we obtain the desired result.
\end{proof}

\begin{cor}
Let $A$ be a finite dimensional $K$-algebra and let $n$ be a natural number. Let $X=X_1\oplus \ldots \oplus X_n$ and $Y=Y_1\oplus \ldots \oplus Y_n$ be finite dimensional multiplicity free tilting modules with $X_{i}$ and $Y_{i}$ indecomposable for all $i=1,\ldots, n$. Then the followings hold:
\begin{itemize}
\item[(i)] If $X_{i}$ is projective, then either $X_{i}\in add Y$ or  there exists a permutation $s$ of $\{1,\ldots,n\}$ such that there is a non-split exact sequence of the form 
\begin{center}
$0\rightarrow X_{i}\rightarrow M\rightarrow Y_{s(i)}\rightarrow 0$
\end{center}
for some $A$-module $M$. 

\item[(ii)] If $X_{i}$ is injective, then either $X_{i}\in add Y$ or  there exists a permutation $s$ of $\{1,\ldots,n\}$ such that there is a non-split exact sequence of the form 
\begin{center}
$0\rightarrow Y_{s(i)}\rightarrow M\rightarrow X_{i}\rightarrow 0$
\end{center}
for some $A$-module $M$.

\item[(iii)] If $X_{i}$ is projective-injective, then $X_{i}\in add Y$ and we have 
$X_{i}\simeq Y_{s(i)}$ for some permutation $s$ of $\{1,\ldots,n\}$.

\end{itemize}

\end{cor}

\section{Examples}
In Theorem \ref{Theorem4} it is proved that the uncommon indecomposable summands of two 
multiplicity free tilting tilting modules appear at the end of some non-split exact sequences. On the other hand, the relation 
between the middle terms is mysterious and complicated. In this section, we provide some examples to illustrate this argument.
As the next example shows, the permutations described in Theorem \ref{Theorem4} are not always unique.
\begin{exam}\label{Ex1} There exist an algebra $A$ and two multiplicity free tilting modules $X=X_1\oplus \ldots \oplus X_4$, $Y=Y_1\oplus \ldots \oplus Y_4$
which satisfy the conditions of Theorem \ref{Theorem4} with the following properties:
\begin{itemize}
\item[(i)] $X$ is projective, $Y$ is injective and $addX\cap addY=0$.
\item[(ii)] There exist an even (resp. odd) permutation $s$ of $\{1,2,3,4\}$ and non-split exact sequences of the form 
\begin{center}
$0\rightarrow X_{i} \rightarrow V_{i}\rightarrow Y_{s(i)} \rightarrow 0$
\end{center}
for any $i=1,\ldots, 4$. Moreover, these sequences are not Auslander-Reiten sequences.
\end{itemize}
\end{exam}
\noindent
{\bf Construction:}
Let $A$ be the algebra given by the quiver

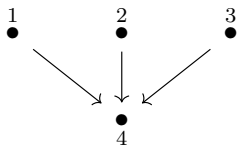
\begin{center}
$\begin{tikzcd} & \overset{1}{\bullet}\ar[rd]&\overset{2}{\bullet} \ar[d]&\overset{3}{\bullet}\ar[ld]\\
 && \underset{4}{\bullet} & &
 \end{tikzcd}$ 
\end{center}
Next, let $X=\begin{matrix}
1\\4
\end{matrix}\oplus\begin{matrix}
2\\4
\end{matrix}\oplus \begin{matrix}
3\\4
\end{matrix}\oplus 4$ and $Y=1\oplus2\oplus3\oplus\begin{matrix}
1~2~3\\
~4~
\end{matrix}$.
Then $(i)$ clearly holds.
Let $\{a,b,c\}=\{1,2,3\}$. We first note that there exist non-split exact sequences (which are not Auslander-Reiten sequences) of the following form:
\begin{itemize}
\item[(1)] $0\rightarrow \begin{matrix}
a\\4
\end{matrix}\rightarrow \begin{matrix}
a~b\\4 \end{matrix} \rightarrow b \rightarrow 0$, $0\rightarrow \begin{matrix}
a\\4
\end{matrix}\rightarrow \begin{matrix}
a~b\\4 \end{matrix}\oplus \begin{matrix}
a~c\\4 \end{matrix} \rightarrow \begin{matrix}
1~2~3\\4 \end{matrix} \rightarrow 0$;
\item[(2)] $0\rightarrow 4 \rightarrow \begin{matrix}
a\\4
\end{matrix}\rightarrow a \rightarrow 0$,  $0\rightarrow 4 \rightarrow \begin{array}{r}
1~2~3\\ 4~4~
\end{array}\rightarrow \begin{matrix}
1~2~3\\4
\end{matrix}\rightarrow 0$.
\end{itemize}
Let $F(1),\ldots, F(4)$ be the subsets defined in Lemma \ref{Lemma2}. Putting together $(i), (1)$ and $(2)$, we conclude that $F(1)=\{2,3,4\}$, $F(2)=\{1,3,4\}$, $F(3)=\{1,2,4\}$ and $F(4)=\{1,2,3,4\}$. Therefore, both the even permutation $(123)$ and the odd permutation $(1234)$ satisfy $(ii)$.
$\square$ \\ 

Note that the four exact sequences describing $(123)$ in the previous example have indecomposable middle term. On the other hand, one of the exact sequences describing $(1234)$ is of the form 
$$0\rightarrow \begin{matrix}
3\\4
\end{matrix}\rightarrow \begin{matrix}
1~3\\4 \end{matrix}\oplus \begin{matrix}
2~3\\4 \end{matrix} \rightarrow \begin{matrix}
1~2~3\\4 \end{matrix} \rightarrow 0.$$
Moreover, neither $(123)$ nor $(1234)$ are described by a reflection in the Auslander-Reiten quiver.

It suffices to consider another Dynkin diagram to find $A,X,Y,s$ as in Theorem \ref{Theorem4} such that the relations between $X_{i}$ and $Y_{s(i)}$ are all those possible a priori. In the following example we show that for tilting modules $X$ and $Y$ that have a common indecomposable summand, there exist non-split exact sequences such that the uncommon indecomposable summands of $X$ and $Y$ appear at the end terms. 

\begin{exam}\label{Ex2}
There exist an algebra $A$, two multiplicity free tilting modules $X=X_1\oplus \ldots \oplus X_4$, $Y=Y_1\oplus \ldots \oplus Y_4$ and $s$ as in Theorem \ref{Theorem4} with the following properties:
\begin{itemize}
\item[(i)] $X_1$ is isomorphic to $Y_{s(1)}$.
\item[(ii)] There exist non-split exact sequences of the form 
$$0\rightarrow X_{i}\rightarrow V_{i}\rightarrow Y_{s(i)}\rightarrow 0$$
with $i=2,4$ and $V_{i}$ is projective.
\item[(iii)] There exists a non-split exact sequence of the form 
$$0\rightarrow Y_{s(3)}\rightarrow V_3 \rightarrow X_3 \rightarrow 0$$ with $V_3$ projective.
\item[(iv)] $s$ is the unique permutation of $\{1,2,3,4\}$ satisfying Theorem \ref{Theorem4}.
\end{itemize}
\end{exam}
\noindent
{\bf Construction:} 
Let $A$ be the algebra given by the quiver 
\begin{center}

$\begin{tikzcd} \underset{1}{\bullet} \ar[r]  & \underset{2}{\bullet}\ar[r] & \underset{3}{\bullet}\ar[r] & \underset{4}{\bullet} \end{tikzcd}$

\end{center}
Next, let $X=\begin{matrix}
1\\2\\3\\4
\end{matrix}\oplus\begin{matrix}
2\\3\\4
\end{matrix}\oplus 2 \oplus 4$ and let $Y=\begin{matrix}
1\\2\\3\\4
\end{matrix}\oplus\begin{matrix}
3\\4
\end{matrix}\oplus 1 \oplus 3$. We first note that \\
$(1)$ $Ext_{A}^1(X,X)=Ext_{A}^1(2,\begin{matrix}
2\\3\\4
\end{matrix}\oplus 2 \oplus 4)$ and $Ext_{A}^1(Y,Y)=Ext_{A}^1(1\oplus 3,\begin{matrix}
3\\4
\end{matrix}\oplus 3)$.
Let $\tau$ denote the Auslander-Reiten translation. Then we have \\ 
\noindent 
$(2)$ $Hom_{A}\big(\begin{matrix}
2\\3\\4
\end{matrix}\oplus 2 \oplus 4, \tau(2)\big)=Hom_{A}\big(\begin{matrix}
2\\3\\4
\end{matrix}\oplus 2 \oplus 4, 3\big)=0$ and $Hom_{A}\big(\begin{matrix}
3\\4
\end{matrix}\oplus 3, \tau(1)\oplus \tau(3)\big)=Hom_{A}\big(\begin{matrix}
3\\4
\end{matrix}\oplus 3, 2\oplus 4\big)=0$. \\
Consequently, we deduce from $(1),(2)$ and \cite[Property (6) of page 76]{Ringel}
that $Ext_{A}^1(X,X)=0$ and $Ext_{A}^1(Y,Y)=0$. Since $A$ is a hereditary algebra with four simple modules, this implies that\\
$(3)$ $X$ and $Y$ are multiplicity free tilting modules such that $X_1=\begin{matrix}
1\\2\\3\\4
\end{matrix}=Y_1$.\\
Moreover, there exist exact sequences of the form \\
$(4)$ $0\rightarrow X_2=\begin{matrix}
2\\3\\4\end{matrix}\rightarrow \begin{matrix}
1\\2\\3\\4
\end{matrix}\rightarrow 1=Y_3\rightarrow 0$ and $0\rightarrow X_4=4\rightarrow \begin{matrix}
3\\4
\end{matrix}\rightarrow 3=Y_4\rightarrow 0$. \\
Finally, there exists an exact sequence of the form \\
$(5)$ $0\rightarrow Y_2=\begin{matrix}
3\\4\end{matrix}\rightarrow \begin{matrix}
2\\3\\4
\end{matrix}\rightarrow 2=X_3\rightarrow 0$.\\
Putting together $(3),(4)$ and $(5)$, we conclude that the modules $X$, $Y$ and the permutation $s=(23)$ satisfy the properties $(i),(ii)$ and $(iii)$. Again, by using the notation $F(i)$ for the following subset $F(i)= \{j \mid \text{either } X_{i}\simeq Y_{j} \text{ or } Ext^1_{A}(X_{i},Y_{j})\oplus 
Ext^1_{A}(Y_{j},X_{i})\neq 0 \}$, we observe that $F(1)=\{1\}$, $F(2)=\{3\}$, $F(3)=\{2,3,4\}$ and 
$F(4)=\{4\}$. Hence, also $(iv)$ holds. 
$\square$  
 
\bigskip
\noindent
The non-split exact sequences constructed in Example \ref{Ex2} have indecomposable
modules in the middle. It is easy to construct examples where this does not happen.

\begin{exam}\label{Ex3}
There exist an algebra $A$ and two multiplicity free tilting modules $X=X_1\oplus X_2 \oplus X_3$, $Y=Y_1\oplus Y_2 \oplus Y_3$ as in Theorem \ref{Theorem4} with the following properties:
\begin{itemize}
\item[(i)] There exists an Auslander-Reiten sequence of the form $$0\rightarrow X_3 \rightarrow V \rightarrow Y_3 \rightarrow 0$$ with $V$ decomposable and $dim_{K}Ext^{1}_{A}(Y_3,X_3)=1$.
\item[(ii)] There exists a unique permutation $s$ of $\{1,2,3\}$ satisfying Theorem \ref{Theorem4} and the map sending $X_{i}$ to $Y_{s(i)}$ for any $i=1,2,3$, is induced by a reflection in the Auslander-Reiten quiver of $A$.  
\end{itemize}
\end{exam}
\noindent
{\bf Construction:} 
Let $A$ be the algebra given by the quiver 
$\begin{tikzcd} \underset{1}{\bullet} \ar[r]  & \underset{2}{\bullet}\ar[r] & \underset{3}{\bullet} \end{tikzcd}$ and let $X$ and $Y$ be the following tilting modules: $X=\begin{matrix}
1\\2\\3\end{matrix}\oplus 2\oplus \begin{matrix}
2\\3\end{matrix}$, $Y=\begin{matrix}
1\\2\\3\end{matrix}\oplus 2\oplus \begin{matrix}
1\\2\end{matrix}$. Then the Auslander-Reiten quiver of $A$ is of the form 
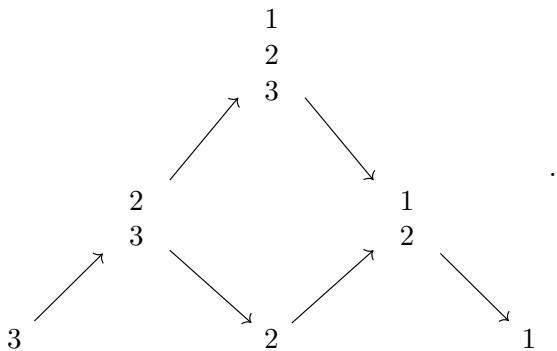
\begin{center}
$\begin{tikzcd}
             &                                                       & \begin{array}{c}1\\2\\3\end{array} \arrow[rd] &                                            &   \\
             & \begin{array}{c}2\\3\end{array} \arrow[ru] \arrow[rd] &                                               & \begin{array}{c}1\\2\end{array} \arrow[rd] &   \\
3 \arrow[ru] &                                                       & 2 \arrow[ru]                                  &                                            & 1
\end{tikzcd}.
$
\end{center}
Consequently, we have $Hom_{A}(\begin{matrix}
2\\3\end{matrix}, \tau(\begin{matrix}
1\\2\end{matrix}))=Hom_{A}(\begin{matrix}
2\\3\end{matrix},\begin{matrix}
2\\3\end{matrix})\simeq K$. Hence, we deduce from \cite[Property (6) of page 76]{Ringel}
that $dim_{K}Ext^{1}_{A}(\begin{matrix}
1\\2\end{matrix},\begin{matrix}
2\\3\end{matrix})=1$. Since 
$0\rightarrow X_3=\begin{matrix}
2\\3\end{matrix}\rightarrow \begin{matrix}
1\\2\\3\end{matrix}\oplus 2\rightarrow \begin{matrix}
1\\2\end{matrix}=Y_3\rightarrow 0$ is an Auslander-Reiten sequence, it follows that $(i)$ holds.
We also note that $X_1=\begin{matrix}
1\\2\\3\end{matrix}=Y_1$, $X_2=2=Y_2$. Let $F(i)$ be the subset defined in Theorem \ref{Theorem4}. Then we clearly have $F(i)=\{i\}$ for $i=1,2,3$. This means that the identity map on $\{1,2,3\}$ is the unique permutation with the desired property. Therefore, also $(ii)$ holds. 
$\square$

\bigskip
\noindent
In the previous examples the dimension of a vector space of the form $Ext_{A}^1(W,U)$, where $\{U,W\}=\{X_{i},Y_{j}\}$ for some $i$ and $j$, is at most two. In the next example, the dimension of some vector space of this form may be bigger than two. To see this, it suffices to use a tame algebra, a projective tilting module, an injective tilting module and some regular modules.

\begin{exam}\label{Ex4} 
There exist an algebra $A$ and two multiplicity free tilting modules $X=X_1\oplus X_2$, $Y=Y_1\oplus Y_2$
as in Theorem \ref{Theorem4} with the following properties:
\begin{itemize}
\item[(i)] Any permutation $s$ of $\{1,2\}$ satisfies Theorem \ref{Theorem4}.
\item[(ii)] $dim_{K} Ext_{A}^1(Y_1,X_2)=2$ and if $0\rightarrow X_2 \rightarrow V \rightarrow Y_1 \rightarrow 0$ is a non-split exact sequence, then $V$ is an indecomposable regular module.
\item[(iii)] $dim_{K} Ext_{A}^1(Y_2,X_1)=4$ and there exist two non-isomorphic indecomposable (res. decomposable)
regular modules $L$ and $M$ and non-split exact sequences of the form $0\rightarrow X_1 \rightarrow V \rightarrow Y_2 \rightarrow 0$ with $V\in \{L,M\}.$
\item[(iv)] $dim_{K} Ext_{A}^1(Y_{i},X_{i})=3$ for any $i=1,2$ and there exist two non-isomorphic indecomposable regular modules $L$ and $M$, a decomposable regular module $N$ and non-split exact sequences of the form  
$0\rightarrow X_{i} \rightarrow V \rightarrow Y_{i} \rightarrow 0$ with $V\in \{L,M, N\}.$
\end{itemize}
\end{exam}
\noindent
{\bf Construction:} 
Let $A$ be the Kronecker algebra (\cite[page 302]{ARS} or \cite[page 122]{Ringel}) given by the quiver 
$\begin{tikzcd}[column sep=huge,row sep=huge]
 \underset{1}{\bullet} \arrow[r, "a"] 
					\arrow[r,"b", swap, shift right=4]
&  \underset{2}{\bullet}  
\end{tikzcd}$ 
Let $X=\begin{matrix}
1\\ 22\end{matrix}\oplus 2$ and let $Y=
1\oplus \begin{matrix}11\\2\end{matrix}$. Then we know from \cite[page 124]{Ringel} that the preinjective component of the Auslander-Reiten quiver of $A$ is of the form \\

$(1)$ \begin{center}$\begin{tikzcd}
\ldots \arrow[rd] \arrow[rd, shift right=2] & &\begin{array}{c} 1 ~ 1 ~ 1~\\ ~ 2 ~ 2 ~ \end{array} \arrow[rd] \arrow[rd, shift right=2] &  &1\\
\ldots                                      & \begin{array}{c} 1 ~ 1 ~ 1 ~ 1\\ ~ 2 ~ 2 ~ 2 ~\end{array} \arrow[ru, shift left=2] \arrow[ru] & & \begin{array}{c} 1 ~ 1\\~ 2 ~\end{array} \arrow[ru]  \arrow[ru, shift left=2] 
\end{tikzcd}.$
\end{center}
Hence, we obtain the following isomorphisms of vector spaces:\\
$(2) Hom_{A}\big(\begin{matrix}
1\\ 22\end{matrix}, \tau(1)\big)\simeq K^3, Hom_{A}(\begin{matrix}
1\\ 22\end{matrix}, \tau\big(\begin{matrix}
11\\ 2\end{matrix})\big)\simeq K^4, Hom_{A}(2, \tau(1))\simeq K^2$ and $Hom_{A}\big(2, \tau\big(
\begin{matrix}11\\2\end{matrix}\big)\big)\simeq K^3$. \\
This observation and \cite[Property $(6)$ of page 76]{Ringel} imply that\\
$(3)$ $dim_{K} Ext_{A}^1\big(1,\begin{matrix}
1\\ 22\end{matrix}\big)=3$, $dim_{K} Ext_{A}^1\big(\begin{matrix}
11\\ 2\end{matrix},\begin{matrix}
1\\ 22\end{matrix}\big)=4$, $dim_{K} Ext_{A}^1(1,2)=2$ and $dim_{K} Ext_{A}^1\big(\begin{matrix}
11\\ 2\end{matrix},2\big)=3$. \newline
Therefore, $(i)$ and $(ii)$ hold. Let $U$ and $W$ be the factor modules of $\begin{matrix}
1\\ 22\end{matrix}$ such that $ann_{A}(U)=<b>$ and $ann_{A}(W)=<a>$ described by the following pictures:
\begin{center}
$U:\begin{tikzcd}[row sep=large,column sep=large] 
\underset{1}{\bullet} \arrow[d,"{a}", no head]\\ 
\underset{2}{\bullet}  
\end{tikzcd}$\ \ \ \ , \ \ \ \ $W:\begin{tikzcd}[row sep=large,column sep=large]
 \underset{1}{\bullet} \arrow[d,"{b}", no head] \\
\underset{2}{\bullet}  
\end{tikzcd}$ .
\end{center}
Then $U$ and $W$ are simple regular modules by \cite[page 124]{Ringel}. 
As usual, for any $n\geq 2$, we denote the following pictures by $U[n]$ and $W[n]$.
Also note that $U[n]$ and $W[n]$ are indecomposable regular modules of dimension $2n$ 
with regular socles $U$ and $W$, respectively.

\begin{center}
$U[n]:\begin{tikzcd}
{\begin{array}{c}1\\ \bullet \end{array}} & {\begin{array}{c}1\\ \bullet \end{array}} & {} & {\begin{array}{c}1\\ \bullet \end{array}} & {\begin{array}{c}1\\ \bullet \end{array}} \\
	& {} & {\ldots\ldots} & {} & {} \\
	{\begin{array}{c}\bullet\\2\end{array}} & {\begin{array}{c}\bullet\\2\end{array}} && {\begin{array}{c}\bullet\\2\end{array}} & {\begin{array}{c}\bullet\\2\end{array}}
	\arrow["{a}"', from=1-1, to=3-1, no head]
	\arrow["{a}", from=1-2, to=3-2, no head]
	\arrow["{a}"', from=1-4, to=3-4, no head]
	\arrow["{a}", from=1-5, to=3-5, no head]
	\arrow["{b}", from=3-1, to=1-2, shift right=1, no head]
	\arrow["{b}", from=3-4, to=1-5, no head]
\end{tikzcd}$.
\end{center}
\begin{center}
$(n \ times)$
\end{center}
\begin{center}
$W[n] :\begin{tikzcd}
{\begin{array}{c}1\\ \bullet \end{array}} & {\begin{array}{c}1\\ \bullet \end{array}} & {} & {\begin{array}{c}1\\ \bullet \end{array}} & {\begin{array}{c}1\\ \bullet \end{array}} \\
	& {} & {\ldots\ldots} & {} & {} \\
	{\begin{array}{c}\bullet\\2\end{array}} & {\begin{array}{c}\bullet\\2\end{array}} && {\begin{array}{c}\bullet\\2\end{array}} & {\begin{array}{c}\bullet\\2\end{array}}
	\arrow["{b}"', from=1-1, to=3-1, no head]
	\arrow["{b}", from=1-2, to=3-2, no head]
	\arrow["{b}"', from=1-4, to=3-4, no head]
	\arrow["{b}", from=1-5, to=3-5, no head]
	\arrow["{a}", from=3-1, to=1-2, shift right=1, no head]
	\arrow["{a}", from=3-4, to=1-5, no head]
\end{tikzcd}$.
\begin{center}
$(n \ times)$
\end{center}
\end{center}
Let $L$ and $M$ denote the following indecomposable (resp., decomposable) modules: $L=U[3], M=W[3]$
(resp., $L=U\oplus W[2], M=W\oplus U[2]$). Then there exist non-split exact sequences of the form 
$0\rightarrow X_1=\begin{matrix}
1\\ 22\end{matrix} \rightarrow V \rightarrow Y_2=\begin{matrix}
11\\ 2\end{matrix}  \rightarrow 0$ with $V\in \{L,M\}$. This observation and $(3)$ imply that $(iii)$ holds. Finally, let $L,M,N$ be the following modules: $L=U[2], M=W[2], N=U\oplus W$. Then there exist non-split exact sequences of the form

$0\rightarrow X_1=\begin{matrix}
1\\ 22\end{matrix} \rightarrow V \rightarrow 1=Y_1\rightarrow 0$ and $0\rightarrow X_2=2 \rightarrow V \rightarrow \begin{matrix}
11\\ 2\end{matrix}=Y_2  \rightarrow 0$ with $V\in \{L,M,N\}$. This remark and $(3)$ imply that also $(iv)$
holds.
$\square$ \\ 

\noindent
The next example explains why the relationship between the indecomposable direct summands of two tilting modules is not always obvious. Indeed, the middle term $V$ of one of the exact sequences in  Theorem \ref{Theorem4} may be the direct sum of three indecomposable projective modules.

\begin{exam}\label{Ex.3.5}
There exist an algebra $A$ and two multiplicity free tilting modules $X=X_1\oplus \ldots \oplus X_6$, $Y=Y_1\oplus \ldots \oplus Y_6$ with $X_{i}$ and $Y_{i}$ indecomposable for any $i=1,\ldots, 6$ and a permutation $s$ of $\{1,\ldots,6\}$ with the following properties:
\begin{itemize}
\item[(i)]$Ext^1_{A}(X_{i},Y_{s(i)})\neq 0$ for any $i=1,\ldots,6.$
\item[(ii)] There are an index $j$ and a non-split exact sequence of the form $0\rightarrow Y_{s(j)}\rightarrow V\rightarrow X_{j}\rightarrow 0$ such that $V$ is the direct sum of three indecomposable projective modules.
\end{itemize}
\end{exam}
\noindent
{\bf Construction:} 
Let $A$ be the algebra given by the quiver 
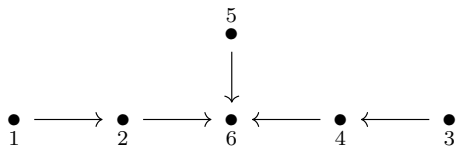
\begin{center}
$\begin{tikzcd} &&\overset{5}{\bullet} \ar[d]&&\\
 \underset{1}{\bullet}\ar[r] & \underset{2}{\bullet}\ar[r]& \underset{6}{\bullet}&
 \underset{4}{\bullet}\ar[l] & \underset{3}{\bullet} \ar[l] 
 \end{tikzcd}$ 
\end{center}
Let $Y$ be the projective generator $Y=\begin{matrix}
1\\2\\6\end{matrix}\oplus \begin{matrix}
2\\ 6\end{matrix}\oplus \begin{matrix}
3\\4\\6\end{matrix}\oplus \begin{matrix}
4\\6\end{matrix}\oplus \begin{matrix}
5\\6\end{matrix}\oplus 6$.
Next, let $X$ be the tilting module constructed in \cite[pages 126-127]{HR}.
Then we have 
$X=\begin{matrix}
5\\6\end{matrix}\oplus \begin{matrix}
25\\6 \end{matrix}\oplus \begin{matrix}
54\\6\end{matrix}\oplus \begin{array}{c} 1 ~ ~ ~ ~ ~ 3\\ ~ 2 ~ 5 ~ 4 ~\\ ~~6~6~~\end{array}\oplus \begin{array}{c} 1 ~ ~ \\  2 ~ 5 \\ ~6~\end{array}\oplus \begin{array}{c} ~ ~ 3\\5~4 ~\\~6~\end{array}.$
Consequently, we have $X_1=\begin{matrix}
5\\ 6\end{matrix}=Y_5$ and there exist non-split exact sequences of the following form:\\
$(1)$ $0\longrightarrow Y_1=\begin{matrix}
1\\2\\6\end{matrix}\longrightarrow \begin{array}{c} 1 ~ ~ ~ ~ ~ 3\\ ~ 2 ~ 5 ~ 4 ~\\ ~~6~6~~\end{array}\longrightarrow \begin{array}{c} ~ ~ ~ ~ 3\\ ~ 5 ~ 4 \\ ~~6~~\end{array}=X_6 \longrightarrow 0$, \\
$(2)$ $0\longrightarrow Y_2=\begin{array}{c} 2\\6\end{array} \longrightarrow \begin{array}{c} 2 ~ 5 ~ 4\\ ~~6~6~~\end{array}\longrightarrow \begin{array}{c} 5 ~ 4\\ ~ 6 ~ \end{array}=X_3\longrightarrow 0$,\\
$(3)$ $0\longrightarrow Y_3=\begin{array}{c} 3\\4\\6\end{array} \longrightarrow \begin{array}{c} 1 ~ ~ ~ ~~3\\ ~2~5~4~\\~~6~6~~\end{array}\longrightarrow \begin{array}{c} 1 ~ ~\\ ~ 2 ~5~\\ ~~6~~ \end{array}=X_5\longrightarrow 0$,\\
$(4)$ $0\longrightarrow Y_4=\begin{array}{c} 4\\6\end{array} \longrightarrow \begin{array}{c} 2 ~ 5 ~ 4\\ ~~6~6~~\end{array}\longrightarrow \begin{array}{c} 2 ~ 5\\ ~ 6 ~ \end{array}=X_2\longrightarrow 0$,\\
$(5)$ $0\longrightarrow Y_6=6 \longrightarrow \begin{array}{c} 1\\2\\6\end{array}\oplus \begin{array}{c} 5\\6\end{array}\oplus \begin{array}{c} 3\\4\\6\end{array}\longrightarrow \begin{array}{c} 1 ~ ~ ~ ~~3\\ ~2~5~4~\\~~6~6~~\end{array}=X_4\longrightarrow 0.$\\
Therefore, the permutation $s=(153246)$ satisfies condition $(i)$. On the other hand condition $(ii)$
follows from $(5)$. $\square$ \\

\noindent
In the previous example, the middle terms of the exact sequences $(1)$, $(3)$ and $(5)$ belong (resp., $(2)$ and $(4)$ do not belong) to  $add (X\oplus Y).$\\

\noindent
{\bf Remark:} 
Happel and Ringels's tilting module constructed in \cite{HR}, that is the tilting $A$-module $X$
in the previous example, seems to be the first example of a non-obvious tilting module.
Since $A$ is hereditary, it follows that $X$ is also a cotilting module. We refer \cite[pages 78-79]{G},
for a naive description of the equivalence and the duality induced by $X$.

In the following example, we show that the permutation in Theorem \ref{Theorem4} is not always cyclic.
\begin{exam}\label{Ex.3.6}
There exist an algebra $A$, two multiplicity free tilting modules $X=X_1\oplus \ldots \oplus X_6$ and $Y=Y_1\oplus \ldots \oplus Y_6$ with $X_{i}$ and $Y_{i}$ indecomposable for any $i=1,\ldots, 6$ and a non-cyclic permutation $t$ of $\{1,\ldots,6\}$ with the following properties:
\begin{itemize}
\item[(i)]$Ext^1_{A}(Y_{t(i)},X_{i})\neq 0$ for any $i=1,\ldots,6.$
\item[(ii)] There exist non-split exact sequences of the form \begin{center}
$0\rightarrow X_{i}\rightarrow V_{i}\rightarrow Y_{t(i)}\rightarrow 0$ \end{center} 
such that either $V_{i}$ is an indecomposable summand of one of the modules $X$ and $Y$, or $V_{i}$ is the direct sum of two 
indecomposable summands of $X$.
\end{itemize}
\end{exam}
\noindent
{\bf Construction:} Let $A$ and $X$ be the algebra and the module constructed in Example \ref{Ex.3.5}. Next let $Y$ be the following injective tilting module:
\begin{center}
$Y=1\oplus\begin{array}{c} 1\\2\end{array}\oplus 3\oplus\begin{array}{c} 3\\4\end{array}\oplus 5\oplus
\begin{array}{c} 1 ~ ~ ~ ~~3\\ ~2~5~4~\\~~6~~\end{array}$.
\end{center} 
Then there exists non-split exact sequences of the following form:\\
$(1)$ $0\longrightarrow X_2=\begin{array}{c} 2 ~ 5\\ ~6~\end{array}\longrightarrow \begin{array}{c} 1 ~ ~ ~ ~\\ ~2~5~\\~6~\end{array}\longrightarrow 1=Y_1\longrightarrow 0$,\\
$(2)$ $0\longrightarrow X_3=\begin{array}{c} 5~ 4\\ ~6~\end{array}\longrightarrow \begin{array}{c} ~ ~ ~ ~3\\ ~5~4~\\~6~\end{array}\longrightarrow 3=Y_3\longrightarrow 0$,\\
$(3)$ $0\longrightarrow X_5=\begin{array}{c} 1~ ~ ~~\\ ~2~5~\\~6~\end{array}\longrightarrow \begin{array}{c} 1 ~ ~ ~ ~~3\\ ~2~5~4~\\~~6~~\end{array}\longrightarrow \begin{array}{c} 3\\4\end{array}=Y_4\longrightarrow 0$,\\
$(4)$ $0\longrightarrow X_6=\begin{array}{c} ~ ~ ~ ~3\\ ~5~4~\\~6~\end{array}\longrightarrow \begin{array}{c} 1 ~ ~ ~ ~~3\\ ~2~5~4~\\~~6~~\end{array}\longrightarrow \begin{array}{c} 1\\2\end{array}=Y_2\longrightarrow 0$,\\
$(5)$ $0\longrightarrow X_1=\begin{array}{c} 5\\6\end{array}\longrightarrow \begin{array}{c} 1 ~ ~ ~ ~\\ ~2~5~\\~6~\end{array}\oplus \begin{array}{c} ~ ~ ~ ~3\\ ~5~4~\\~6~\end{array} \longrightarrow \begin{array}{c} 1 ~ ~ ~ ~~3\\ ~2~5~4~\\~~6~~\end{array}=Y_6\longrightarrow 0$,\\
$(6)$ $0\longrightarrow X_4=\begin{array}{c} 1 ~ ~ ~ ~~3\\ ~2~5~4~\\~~6~6~~\end{array}\longrightarrow \begin{array}{c} 1 ~ ~ ~ ~\\ ~2~5~\\~6~\end{array}\oplus \begin{array}{c} ~ ~ ~ ~3\\ ~5~4~\\~6~\end{array} \longrightarrow 5=Y_5\longrightarrow 0$.\\
Let $t$ denote the permutation $(162)(45)$. Then comparing $(1),\ldots,(6)$, we conclude that $X$, $Y$ and $t$ satisfy conditions $(i)$ and $(ii)$. $\square$ \\

\bigskip
\noindent
In the next example, we construct two tilting modules $X$ and $Y$ with
\begin{itemize}
\item isomorphic injective envelopes;
\item isomorphic endomorphism rings;
\item the same injective summands, but different projective summands;
\item the same projective dimension, but different injective dimensions. 
\end{itemize}
We also note that in the next example there are exact sequences as in Theorem \ref{Theorem4} whose middle terms belong to $addX\cap addY.$
\bigskip

\begin{exam}\label{Ex.3.7}
There exist $K$-algebras $A$ and $\Lambda$ and multiplicity free tilting $A$-modules $X=X_1\oplus \ldots \oplus X_5$ and 
$Y=Y_1\oplus \ldots \oplus Y_5$ with $X_i$ and $Y_i$ indecomposable for any $i=1,\ldots, 5$ with the following properties:
\begin{itemize}
\item[(i)]$X$ and $Y$ are multiplicity free tilting modules with the same injective envelopes.
\item[(ii)] $X\subseteq Y$ and $Y/X$ is semisimple.
\item[(iii)] $X$ is a cotilting module and $idim Y=2$.
\item[(iv)] There is a permutation $s$ of $\{1,\ldots, 5\}$ such that either $X_{i}\simeq Y_{s(i)}$
or there exists a non-split exact sequence of the form
\begin{center}
$0\longrightarrow Y_{s(i)} \longrightarrow V_{i} \longrightarrow X_{i} \longrightarrow 0$
\end{center}
with $V_{i}\in add X \cap add Y$.
\item[(v)] There is a direct summand $U$ of $X$ and $Y$ such that $GenU=GenX\subsetneqq Gen Y$,
$CogenU=CogenX=CogenY$, and we have $Gen U\subseteq Gen T$ and $Cogen U\subseteq Cogen T$ for any
tilting $A$-module $T$.
\item[(vi)] $End_{A}(X)$ and $End_{A}(Y)$ are isomorphic to $\Lambda$.
\end{itemize}
\end{exam}
\noindent
{\bf Construction:} Let $A$ be the algebra given by the quiver
$$
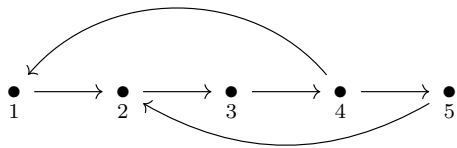
\begin{tikzcd}
\underset{1}{\bullet} \arrow[r] & \underset{2}{\bullet} \arrow[r] & \underset{3}{\bullet} \arrow[r] & \underset{4}{\bullet} \arrow[r] \arrow[lll, bend right=50]& \underset{5}{\bullet} \arrow[lll, bend left=30]
\end{tikzcd}$$
with relations such that the indecomposable projective left $A$-modules have the following form:
\begin{center}
$ \begin{array}{c} 1\\2\\3\\4\\1 \end{array}, \begin{array}{c} 2\\3\\4\\1 \end{array},
\begin{array}{c} 3\\4\\1 \end{array}, \begin{array}{c} ~4~\\1~5\\~2~ \end{array}, \begin{array}{c} 5\\2 \end{array}$.
\end{center}
Let $X$ and $Y$ denote the following modules:
\begin{center}
$X=\begin{array}{c} 1\\2\\3\\4\\1 \end{array} \oplus \begin{array}{c} 1\\2 \end{array}\oplus 
\begin{array}{c} ~4~\\1~5\\~2~ \end{array}\oplus 1 \oplus \begin{array}{c} 4\\1 \end{array}$ and 
$Y=\begin{array}{c} 1\\2\\3\\4\\1 \end{array} \oplus \begin{array}{c} 1~ 5\\~2~ \end{array}\oplus 
\begin{array}{c} ~4~\\1~5\\~2~ \end{array}\oplus 1 \oplus \begin{array}{c} 3\\4\\1 \end{array}$. 
\end{center} 
Then $X$ and $Y$ are direct sums of five pairwise non-isomorphic modules of projective dimension $\leq 1$.
We also note that \\
$(1)$ $Ext^{1}_{A}(X,X)=Ext^{1}_{A}(\begin{array}{c} 1\\2 \end{array}\oplus 1 \oplus \begin{array}{c}   4\\1 \end{array}, \begin{array}{c} 1\\2 \end{array}\oplus 1 \oplus \begin{array}{c}   4\\1 \end{array})$
and $Ext^{1}_{A}(Y,Y)=Ext^{1}_{A}(\begin{array}{c} 1~ 5\\~2~ \end{array}\oplus 1, \begin{array}{c} 1~ 5\\~2~ \end{array}\oplus 1 \oplus \begin{array}{c} 3\\4\\1 \end{array})$.\\
On the other hand, we have \\
$(2)$ $Hom_{A}\big(\begin{array}{c} 1\\2 \end{array}\oplus 1 \oplus \begin{array}{c} 4\\1 \end{array}, \tau(\begin{array}{c} 1\\2 \end{array})\oplus \tau(1) \oplus \tau(\begin{array}{c} 4\\1 \end{array})\big)
= Hom_{A}\big(\begin{array}{c} 1\\2 \end{array}\oplus 1 \oplus \begin{array}{c} 4\\1 \end{array}, \begin{array}{c} 2\\3 \end{array}\oplus \begin{array}{c} 5\\2 \end{array} \oplus \begin{array}{c} 5 \end{array}\big)=0$ \ and \ $Hom_{A}\big(\begin{array}{c} 1~ 5\\~2~ \end{array}\oplus 1 \oplus \begin{array}{c} 3\\4\\1 \end{array},\tau(\begin{array}{c} 1~ 5\\~2~ \end{array})\oplus \tau(1)\big)= 
Hom_{A}\big(\begin{array}{c} 1~ 5\\~2~ \end{array}\oplus 1 \oplus \begin{array}{c} 3\\4\\1 \end{array},\begin{array}{c} 2 \end{array}\oplus \begin{array}{c} 5\\2 \end{array}\big)=0$.\\
Comparing $(1)$ and $(2)$, we deduce from \cite[Proposition (6) of page 76]{Ringel} that $Ext^{1}_{A}(X,X)=0$ and $Ext^{1}_{A}(Y,Y)=0$. This implies that \\
$(3)$ $X$ and $Y$ are multiplicity free tilting modules.\\
We also note that $X\subseteq Y$ and $Y/X\simeq 3 \oplus 5$ and the injective envelopes of $X$
and $Y$ are both isomorphic to the injective envelopes of $1\oplus 1\oplus 1\oplus 2 \oplus 2$.
This remark and $(3)$ imply that $(i)$ and $(ii)$ hold. Since the indecomposable injective modules are
\begin{center}
$ \begin{array}{c} 1\\2\\3\\4\\1 \end{array}, \begin{array}{c} ~4~\\1~5\\~2~ \end{array}, \begin{array}{c} 1\\2\\3 \end{array},
\begin{array}{c} 1\\2\\3\\4 \end{array}, \begin{array}{c} 4\\5 \end{array}$,
\end{center}
it immediately follows that $idim X=1$ and $idim Y=2$. Hence, $(iii)$
follows from $(3)$ and from the definition of cotilting module. We also note that the following sequences 
are exact: \\
$(4)$
$0\rightarrow Y_2=\begin{array}{c} 1~ 5\\~2~ \end{array} \rightarrow \begin{array}{c} ~4~\\ 1~ 5\\~2~ \end{array}\oplus 1 \rightarrow \begin{array}{c} 4\\1 \end{array}=X_5\rightarrow 0$, \\
\noindent
$(5)$ $0\rightarrow Y_5=\begin{array}{c} 3\\4\\1 \end{array} \rightarrow \begin{array}{c} 1\\2\\3\\4\\1 \end{array} \rightarrow \begin{array}{c} 1\\2 \end{array}=X_2\rightarrow 0$.\\
The definition of $X$ and $Y$, $(4)$ and $(5)$ imply that $(iv)$ holds with $s=(25)$. 
Let $U$ be the module
$\begin{array}{c} 1\\2\\3\\4\\1 \end{array}\oplus
\begin{array}{c} ~4~\\1~5\\~2~ \end{array}$. Then there exist obvious epimorphisms $U=X_1\oplus X_3 \rightarrow X_{i}$ with $i=2,4,5.$ On the other hand, the projective module $Y_5=\begin{array}{c} 3\\4\\1 \end{array}$ is not a summand of $U$. Hence, we clearly have $GenU=GenX\neq GenY$ and $CogenU=CogenX=CogenY$. Since $U$ is a direct summand of any tilting module, we conclude that $(v)$ holds. 

Finally, let $Z_{i}=X_{i}$(resp., $Z_{i}=Y_{i}$) for any $i=1,\ldots, 5.$ Then we have $End_{A}(Z_1)\simeq k[x]/(x^2)$, $End_{A}(Z_{i})\simeq k$ for $i>1$ and $dim_{K} Hom_{A}(Z_{i},Z_{j})\leq 1$ if $i\neq j$. Moreover, a nonzero vector space of the form $Hom_{A}(Z_{i},Z_{j})$ with $i\neq j$ is generated by a suitable composition of the maps $f,g,h,l,m,p$ described by the following picture 
$$
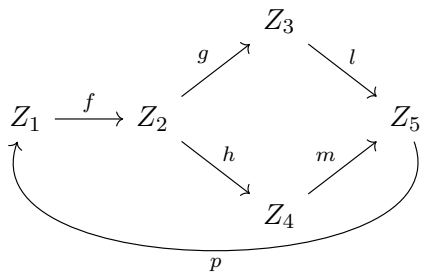
\begin{tikzcd}
                                                       &                                                                                                 & Z_3 \arrow[rd,"l"] &                                                                                      \\
Z_1 \arrow[r,"f"] & Z_2 \arrow[ru,"g"] \arrow[rd,"h"] &                                                                                & Z_5\arrow[lll,"p", bend left=110] \\
                                                       &                                                                                                 & Z_4\arrow[ru,"m"] &
\end{tikzcd}$$
with the property that $lg=mh$ and $fplgf=0=fpmhf.$ It follows that $End_{A}(X)$ and $End_{A}(Y)$ are isomorphic to the algebra $\Lambda$ given by the quiver
$$
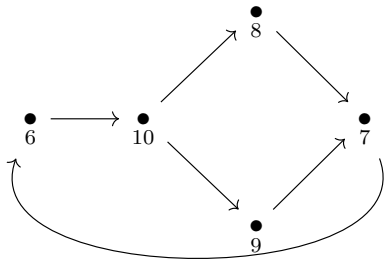
\begin{tikzcd}
                                                       &                                                                                                 & \underset{8}{\bullet} \arrow[rd] &                                                                                      \\
\underset{6}{\bullet} \arrow[r] & \underset{10}{\bullet} \arrow[ru] \arrow[rd] &                                                                                & \underset{7}{\bullet}\arrow[lll, bend left=110] \\
                                                       &                                                                                                 & \underset{9}{\bullet}\arrow[ru] &
\end{tikzcd}$$
with relations such that the indecomposable projective $\Lambda$-modules have the following form 
\begin{center}
$\begin{array}{c} 6\\10 \\8~9\\7\\6 \end{array},\begin{array}{c} 7\\6 \end{array}, 
\begin{array}{c} 8\\7\\6 \end{array}, \begin{array}{c} 9\\7\\6 \end{array}, \begin{array}{c} 10 \\8~9\\7\\6 \end{array}$. 
\end{center}
Hence, also $(iv)$ holds. $\square$ \\

\noindent
\textbf{Acknowledgements:} The authors would like to thank warmly Professor Luise Unger. Indeed, she informed them that Happel's book \cite{HR3} contains the proof of the equivalence of $(\gamma)$ and $(\gamma')$ for a partial tilting module. 

The authors also would like to thank the referee for carefully reading the manuscript and for valuable comments.




     
    


\end{document}